\documentclass{article}
\usepackage{amsmath,amssymb,amsthm}
\usepackage{enumitem}
\usepackage[numbers]{natbib}
\usepackage[colorlinks=true,linkcolor=blue,citecolor=blue,urlcolor=blue]{hyperref}
\usepackage{doi}
\theoremstyle{plain}
\DeclareMathOperator{\idx}{idx}
\newtheorem{mainthm}{Theorem}

\newtheorem{theorem}{Theorem}[section]
\newtheorem{proposition}[theorem]{Proposition}
\newtheorem{lemma}[theorem]{Lemma}
\newtheorem{corollary}[theorem]{Corollary}
\theoremstyle{definition}
\newtheorem{definition}[theorem]{Definition}
\newtheorem{example}[theorem]{Example}
\newcommand{\Z}{\mathbb{Z}}
\newcommand{\F}{\mathbb{F}}
\newcommand{\N}{\mathbb{N}}
\newcommand{\Aut}{\operatorname{Aut}}
\newcommand{\PAut}{\operatorname{PAut}}
\newcommand{\Sym}{\operatorname{Sym}}
\newcommand{\Aff}{\operatorname{Aff}}
\newcommand{\AGL}{\operatorname{AGL}}
\newcommand{\Cyc}{\operatorname{Cyc}}
\newcommand{\rk}{\operatorname{rk}}
\newcommand{\ord}{\operatorname{ord}}
\DeclareMathOperator{\Tr}{Tr}
\newcommand{\gen}[1]{\langle #1\rangle}
\newcommand{\cC}{\mathcal{C}}
\newcommand{\cP}{\mathcal{P}}
\newcommand{\cQ}{\mathcal{Q}}
\newcommand{\cS}{\mathcal{S}}
\newcommand{\one}{\mathbf{1}}

\title{Permutation automorphism groups of cyclic codes I: cyclotomic association schemes}

\author{
  Yanni Wu \thanks{Email: \texttt{12031209@mail.sustech.edu.cn}} \quad and \quad Ziqing Xiang \thanks{Email: \texttt{xiangzq@sustech.edu.cn}} \\
  Southern University of Science and Technology \\
}

\date{}

\begin{document}

\maketitle

\begin{abstract}
The Berger-Charpin conjecture predicts that the permutation automorphism group of a cyclic code is generally the $q$-affine group, which is generated by the shift and the Frobenius multiplier on the index set. The word {\em generally} is not made precise in the literature, and the goal of this series of papers is to study when the Berger-Charpin conjecture or its variants hold.

In this first part, we establish a natural connection between the theory of cyclic codes and the theory of cyclotomic association schemes (equivalently, Schur rings over cyclic groups). Using it, we show that the $q$-affine group is not the correct group to expect for certain code lengths. We characterize these lengths by an arithmetic condition, and show that they have density zero.
\end{abstract}

\section{Introduction}\label{sec:intro}

When the code length $n$ and field size $q$ are coprime, which we assume throughout the paper, every cyclic code is invariant under the shift and the Frobenius multiplier (multiplication by $q$) on the index set. These two permutations generate the {\em $q$-affine group} $\Aff_q(n) := \Z_n \rtimes \gen q$, which is always a subgroup of the {\em permutation automorphism group} $\PAut(C)$ of a cyclic code $C$ of length $n$ over $\F_q$.

In 1996, Berger-Charpin \cite{BergerCharpin1996} characterized $\PAut(C)$ for affine-invariant extended cyclic codes $C$, and determined it for the primitive BCH codes over a prime field. Based on this work and other evidence, Charpin \cite{Charpin1998} conjectured that the permutation automorphism group of cyclic codes is generally the $q$-affine group. This is known as the {\em Berger-Charpin conjecture}. It is the coding-theoretic counterpart of the Babai-Godsil theorem \cite{BabaiGodsil1982} in 1982, which shows that almost all Cayley graphs have no automorphisms beyond the obvious ones.

Evidence for the Berger-Charpin conjecture has been accumulating. In 1993, Huffman-Job-Pless \cite{HuffmanJobPless1993} analyzed the multipliers of cyclic objects and of cyclic codes. In 2010, Bienert-Klopsch \cite{BienertKlopsch2010} determined the binary cyclic codes whose permutation group is primitive, which settles the Berger-Charpin conjecture for prime length. In 2013, Guenda-Gulliver \cite{GuendaGulliver2013} classified the permutation groups of cyclic codes of prime length over an arbitrary finite field.

The strongest evidence to date is due to Feng-Hollmann-Li-Xiang \cite{Feng2026} in 2026. That work used the classification of finite simple groups to determine every non-degenerate irreducible cyclic code whose permutation automorphism group is not the $q$-affine group. Such codes can all be obtained from a few explicitly described families and their descendants under certain secondary constructions.

Our first main result shows, however, that the group to be expected in general is not the $q$-affine group $\Aff_q(n)$, but the automorphism group of a cyclotomic association scheme.

\begin{mainthm}\label{thm:A}
Let $\Aut_q(n):=\Aut\big(\Cyc(\gen q,\Z_n)\big)$, where $\Cyc(\gen q,\Z_n)$ is the cyclotomic association scheme of the subgroup $\gen q\le\Z_n^\times$ acting on $\Z_n$ by multiplication, defined in Section~\ref{sec:cyc}. Then,
\[
	\bigcap_{C} \PAut(C) \;=\; \Aut_q(n),
\]
where $C$ runs over all cyclic codes of length $n$ over $\F_q$.
\end{mainthm}

{\em Cyclotomic cosets} play a fundamental role in the study of cyclic codes. They also generate an algebraic structure, called a {\em cyclotomic association scheme}, or equivalently, a {\em Schur ring over a cyclic group}. This apparently new connection between cyclic codes and cyclotomic association schemes is the basis of Theorem~\ref{thm:A}. We establish the connection and prove Theorem~\ref{thm:A} in Section~\ref{sec:link}.

The group $\Aut_q(n)$ has been studied in the theory of Schur rings. In 1933, Schur \cite{Schur1933} introduced what later came to be called {\em Schur rings}. In 1964, Wielandt \cite{Wielandt1964} developed them in order to study permutation groups. Schur rings over a cyclic group have been classified. A series of works by Leung-Ma \cite{LeungMa1990} in 1990, Muzychuk \cite{Muzychuk1994} in 1994, and Leung-Man \cite{LeungMan1996,LeungMan1998} in 1996 and 1998 showed that every such Schur ring is built from smaller ones by iterated tensor and generalized wreath products. Building on the above structural results, in 2003, Evdokimov-Ponomarenko \cite{EvdokimovPonomarenko2003} characterized the pairs $(n,q)$ for which $\Aut_q(n)=\Aff_q(n)$, and in 2008 they partially generalized it to cyclotomic association schemes over finite commutative rings \cite{EvdokimovPonomarenko2008}.

We call $n$ an {\em exceptional length} for $q$ if $\Aut_q(n) \neq \Aff_q(n)$. The Evdokimov-Ponomarenko characterization \cite{EvdokimovPonomarenko2003} is a structural characterization of exceptional lengths. Our second main result is an arithmetic one. It is deduced from and simplifies the Evdokimov-Ponomarenko characterization.

\begin{mainthm}\label{thm:B}
Let $\idx_d(q):=\frac{\varphi(d)}{\ord_d(q)}$. Then, $\Aut_q(n)\ne\Aff_q(n)$ if and only if one of the following holds.
\begin{enumerate}[label=\textup{(E\arabic*)},leftmargin=*]
\item\label{it:E1} $\idx_{n/p}(q)=\idx_n(q)$ for some prime $p\ge5$ with $p\,\|\,n$.
\item\label{it:E2} $\idx_{n/p}(q)=\idx_n(q)$ for some odd prime $p$ with $p^{2}\mid n$.
\item\label{it:E3} $\idx_{n/4}(q)=\idx_n(q)$ and $8\mid n$.
\end{enumerate}
\end{mainthm}

Theorem~\ref{thm:B} is proved in Section~\ref{sec:examples}. One can check that the following $n$ are exceptional for the corresponding $q$.
\begin{enumerate}[label=\textup{(\alph*)}]
\item $q=2$ and $n = 5 (2^k - 1)$ with $k\ge3$ odd.
\item $q=2$ and $n = 3^k$ with $k \geq 2$.
\item $q=5$ and $n = 2^k$ with $k \geq 4$.
\end{enumerate}
These three families correspond to Conditions~\ref{it:E1}, \ref{it:E2} and~\ref{it:E3}, respectively, and they are Examples~\ref{ex:famA}, \ref{ex:famB} and~\ref{ex:famC}.

Proposition~\ref{prop:prim} proves that no length of the form $n = q^{m}-1$ is exceptional for $q$. Primitive BCH codes and Reed-Solomon codes all have lengths of this form.

The third main result of this paper is on the density of exceptional lengths.

\begin{mainthm}\label{thm:C}
Fix a prime power $q$. Then, the exceptional lengths $n$ for $q$ have density zero. In other words,
\[
\#\{n \in \N: n \leq x, \ \gcd(n,q)=1\ \text{and}\ \Aut_q(n)\neq\Aff_q(n)\}=o(x),
\]
as $x$ goes to infinity.
\end{mainthm}

The proof of Theorem~\ref{thm:C} is given in Section~\ref{sec:density}. It is obtained by analyzing how often Conditions~\ref{it:E1}, \ref{it:E2} and~\ref{it:E3} hold.

The paper is organized as follows. The theory of cyclotomic association schemes and the theory of cyclic codes are reviewed in Section~\ref{sec:cyc} and Section~\ref{sec:codes}, respectively. Section~\ref{sec:link} establishes a link between them, and proves Theorem~\ref{thm:A}. Section~\ref{sec:examples} turns to exceptional lengths and proves Theorem~\ref{thm:B}. Finally, Section~\ref{sec:density} bounds how often such lengths occur, and proves Theorem~\ref{thm:C}.

\subsection*{Notation}\label{ssec:notation}
Throughout the paper, assume that $n$ and $q$ are coprime. Let $\N$ be the set of positive integers, let $\Z_n := \Z/n\Z$, and let $\Z_n^\times$ be its group of units, which is of order $\varphi(n)$. For a positive integer $d$ coprime to $q$, the order of $q$ in $\Z_d^\times$ is denoted by $\ord_d(q)$, and let $\idx_d(q):=\varphi(d)/\ord_d(q)$. For elements $g_1,\dots,g_r$ of a group, $\gen{g_1,\dots,g_r}$ is the subgroup that they generate. For a prime $p$ and a nonzero integer $x$, we write $v_p(x)$ for the exponent of $p$ in $x$, and $p^{a}\,\|\,x$ to mean that $v_p(x)=a$. We write $f\ll g$, or equivalently $g\gg f$, to mean that $f\le Cg$ for a positive constant $C$ that is allowed to depend on $q$.

\section{Cyclotomic association schemes}\label{sec:cyc}

The concept of an association scheme, or one of its variants, was discovered independently several times. Two of those lines of development converge on the theory used here.

The first is the theory of association schemes. The name {\em association scheme} was coined by Bose-Shimamoto \cite{BoseShimamoto1952} in 1952, and Bose-Mesner \cite{BoseMesner1959} identified their algebraic form in 1959. In 1973, Delsarte \cite{Delsarte1973} made them into a systematic tool for coding theory and design theory, and introduced the {\em cyclotomic association schemes over finite fields}. Their counterparts over finite rings were studied by Goldbach-Claasen \cite{GoldbachClaasen1992,GoldbachClaasen1992a} in 1992.

The second is the theory of {\em Schur rings}, recalled in Section~\ref{sec:intro}, which runs from Schur \cite{Schur1933} and Wielandt \cite{Wielandt1964} to the classification over a cyclic group in the 1990s. That the two theories describe the same objects was recognized only later, and Evdokimov-Ponomarenko \cite{EvdokimovPonomarenko2003,EvdokimovPonomarenko2008} carried the combined theory further in the 2000s.

In this paper, we choose to use the language of cyclotomic association schemes. This section presents the basic definitions, keeping only what is necessary for our purposes. A reader seeking the full theory may consult Bannai-Ito \cite{BannaiIto1984} or Zieschang \cite{Zieschang2005} for association schemes, and the survey of Muzychuk-Ponomarenko \cite{MuzychukPonomarenko2009} for Schur rings. Our schemes are the homogeneous coherent configurations, a class introduced independently by Higman \cite{Higman1971} and by Weisfeiler-Leman \cite{WeisfeilerLehman1968}, for which we refer to Chen-Ponomarenko \cite{ChenPonomarenko2019}.

\subsection{Colorings and schemes}\label{ssec:defs}

The definition given here is slightly different from the classical ones in \cite{BannaiIto1984,Zieschang2005}, but is equivalent to those. We present a scheme as a coloring of ordered pairs, rather than as a partition of $X\times X$.

A \emph{color set} is a triple $\Sigma=(S,0,{}^*)$ consisting of a finite set $S$, an element $0\in S$, and an involution $s\mapsto s^*$ of $S$ fixing $0$.

\begin{definition}\label{def:scheme}
Let $\Sigma$ be a color set, and $X$ a finite set. An \emph{association scheme} on $X$ over $\Sigma$ is a pair $\cC=(X,\pi)$ with $\pi\colon X\times X\to\Sigma$ satisfying the following axioms:
\begin{enumerate}[label=\textup{(A\arabic*)},leftmargin=*]
\item\label{ax:surj} $\pi$ is surjective;
\item\label{ax:diag} $\pi(x,y)=0$ if and only if $x=y$;
\item\label{ax:star} $\pi(y,x)=\pi(x,y)^*$ for all $x,y\in X$;
\item\label{ax:reg} for all $r,s,t\in\Sigma$ there is an integer $c^t_{rs}$ such that
\[
\#\{y\in X:\pi(x,y)=r,\ \pi(y,z)=s\}=c^t_{rs},
\qquad\text{whenever }\pi(x,z)=t.
\]
\end{enumerate}
Here $\Sigma(\cC) := \Sigma$ is the \emph{color set of $\cC$}, the integers $c^t_{rs}$ are the \emph{structure constants}, and $\rk(\cC):=|\Sigma(\cC)|$ is the \emph{rank} of $\cC$.
\end{definition}

The \emph{automorphism group} of $\cC$ is
\[
\Aut(\cC):=\{\sigma\in\Sym(X):\pi\circ(\sigma\times\sigma)=\pi\},
\]
where $\Sym(X)$ is the group of all bijections $X \to X$.

\subsection{The scheme algebra}
The structure constants make the color set into the basis of an algebra.

\begin{definition}\label{def:schemealg}
Let $k$ be a commutative ring and $\cC$ an association scheme over $\Sigma$. The \emph{scheme algebra} $k\cC$ is the free $k$-module with basis $\Sigma$, with multiplication determined by
\[
s\cdot t := \sum_{u\in\Sigma}c^u_{st}\,u.
\]
\end{definition}

The axioms of association schemes ensure that the scheme algebra $k\cC$ is an associative $k$-algebra with identity, free of rank $\rk(\cC)$.

\subsection{The cyclotomic association scheme of a subgroup}\label{ssec:cycdef}

Let $R$ be a finite ring with identity, which is not assumed to be commutative. Let $K\le R^\times$ be a subgroup of the unit group $R^\times$ of $R$, acting on $R$ by right multiplication. Write $aK:=\{ak:k\in K\}$ for the orbit of $a\in R$, and write $R/K$ for the set of orbits. Since $0K=\{0\}$ and $(-a)k=-(ak)$, the set $R/K$ is a color set with distinguished element $\{0\}$ and involution $(aK)^\ast:=(-a)K$.

\begin{definition}\label{def:cyc}
The \emph{cyclotomic association scheme} of $K$ over $R$ is the pair $\Cyc(K,R):=(R,\pi)$ over the color set $R/K$, where
\[
\pi(x,y) := (y-x)K.
\]
\end{definition}
It is direct to verify that this is an association scheme. The next proposition identifies its affine automorphisms.

\begin{proposition}\label{prop:cycscheme}
Let
\begin{gather*}
\AGL_1(R):=\{x\mapsto xa+b:\ a\in R^\times,\ b\in R\},\\
\Aff_K(R):=\{x\mapsto xk+b:\ k\in K,\ b\in R\}.
\end{gather*}
Then, $\Aut\big(\Cyc(K,R)\big)\cap\AGL_1(R)=\Aff_K(R)$.
\end{proposition}

\begin{proof}
Consider a map $\sigma\colon x\mapsto xa+b$ with $a\in R^\times$ and $b\in R$. Then $\sigma y-\sigma x=(y-x)a$, so
\[
\pi(\sigma x,\sigma y)=(y-x)aK,
\]
and $\sigma$ preserves the coloring if and only if $(y-x)aK=(y-x)K$ for all $x$ and $y$. Taking $y-x=1$ shows that this forces $aK=K$, that is, $a\in K$; and conversely $aK=K$ gives $(y-x)aK=(y-x)K$ for every $x$ and $y$. Hence a map in $\AGL_1(R)$ preserves the coloring if and only if $a\in K$, which is the stated equality.
\end{proof}

\section{Cyclic codes}\label{sec:codes}

There are several equivalent ways to define cyclic codes. They can be presented as shift-invariant subspaces of $\F_q^n$, as ideals of the quotient ring $\F_q[x]/(x^n-1)$, as ideals of the group algebra $\F_q\Z_n$, a point of view introduced for abelian codes by Berman \cite{Berman1967} in 1967 and by MacWilliams \cite{MacWilliams1970} in 1970, or as ideals of the function space $\F_q^{\Z_n}$ under convolution. We choose the last one, which is the dual of the group algebra viewpoint, since it fits best with the theory of cyclotomic association schemes.

In this section, we review some classical concepts and results in the theory of cyclic codes. We give only statements, and readers can refer to Huffman-Pless \cite[Ch.~4]{HuffmanPless2003} or MacWilliams-Sloane \cite[Ch.~7--8]{MacWilliamsSloane1977} for detailed proofs.

Let $\Omega := \Z_n$. Then, $\F_q^\Omega$ is an $\F_q$-algebra under the convolution product $\ast$, defined by
\[
(f*g)(x):=\sum_{y\in\Omega}f(x-y)g(y)\qquad \text{for $f, g \in \F_q^\Omega$.}
\]
We have the identification of algebras $\F_q[x]/(x^n-1)\ \xrightarrow{\ \sim\ }\ \F_q^\Omega$ given by
\begin{equation}\label{eq:poly}
\sum_ic_ix^i\longmapsto (i \mapsto c_i).
\end{equation}
For a subset $S$ of a set $X$, let $\one_S\colon X\to\F_q$ be its indicator function, and for $b\in\Omega$, let $\delta_b:=\one_{\{b\}}$. Then, $\delta_0$ is the identity of the convolution product, and left-multiplication by $\delta_1$ is the cyclic shift. So, a subspace of $\F_q^\Omega$ is an ideal precisely when it is closed under the shift.

\begin{definition}\label{def:cyccode}
A \emph{cyclic code} of length $n$ over $\F_q$ is an ideal of $\F_q^\Omega$.
\end{definition}

\subsection{Automorphisms}\label{ssec:perms}
The symmetric group $\Sym(\Omega)$ acts on $\F_q^\Omega$ by $(\sigma c)(i):=c(\sigma^{-1}i)$, and the \emph{permutation automorphism group} of a code $C$ is
\[
\PAut(C):=\{\sigma\in\Sym(\Omega):\sigma C=C\}.
\]
The {\em translations} $\tau_b\colon i\mapsto i+b$ for $b\in\Omega$ and the {\em multiplier} $\mu_q\colon i\mapsto qi$ are always permutation automorphisms. Since $\tau_1$ generates all the translations, the group they generate is
\[
\Aff_q(n):=\gen{\tau_1,\ \mu_q}=\Z_n\rtimes\gen q\ \le\ \Sym(\Omega).
\]
Since $\mu_q$ is an automorphism of the group $\Z_n$, its action on functions is an automorphism of the algebra $\F_q^\Omega$, and under Eq.~\eqref{eq:poly} that action is the Frobenius map $x\mapsto x^q$ in $\F_q[x]/(x^n-1)$.

\subsection{Characters}\label{ssec:chars}

Cyclic codes are closely related to characters, which we now set up.

Let $F:=\F_{q^{\ord_n(q)}}$, the smallest extension of $\F_q$ containing a primitive $n$-th root of unity, and fix such a root $\zeta$. For $a\in\Omega$ let $\chi_a\in F^\Omega$ be the character $\chi_a(i)=\zeta^{-ai}$, and extend the action of $\Sym(\Omega)$ to $F^\Omega$ by the same formula. Since $n$ is invertible in $F$, and $\sum_{i}\zeta^{ai}$ equals $n$ for $a=0$ and $0$ otherwise, we have the following.
\begin{enumerate}[label=\textup{(C\arabic*)},leftmargin=*]
\item $\{\chi_a\}_{a\in\Omega}$ is an $F$-basis of $F^\Omega$;
\item $\chi_a*\chi_a=n\chi_a$ and $\chi_a*\chi_b=0$ for $a\ne b$, so the $\frac1n\chi_a$ are orthogonal idempotents summing to $\delta_0$, and $F^\Omega=\bigoplus_aF\chi_a$ is the decomposition into minimal ideals;
\item\label{ch:mu} $\mu_q\chi_a=\chi_{q^{-1}a}$.
\end{enumerate}
For a subspace $C\subseteq\F_q^\Omega$, let $C\otimes F$ denote its $F$-span in $F^\Omega$.

The two subsections that follow record the classical results we need: a cyclic code is determined by a $\gen q$-invariant set of characters, it is generated by an idempotent, and it is a direct sum of minimal ones.

\subsection{Cyclotomic cosets and the spectrum}\label{ssec:spec}

The orbits of $\gen q$ acting on $\Z_n$ by multiplication are the {\em $q$-cyclotomic cosets}. We write $T$ for such an orbit, regarded as an element of the orbit set $\Z_n/\gen q$.

\begin{proposition}\label{prop:spectrum}
For every cyclic code $C$ of length $n$ over $\F_q$ there is a unique subset $J(C)\subseteq\Omega$ with $C\otimes F=\bigoplus_{a\in J(C)}F\chi_a$, and $J(C)$ is $\gen q$-invariant. The map $C\mapsto J(C)$ is an inclusion-preserving bijection from the cyclic codes of length $n$ onto the $\gen q$-invariant subsets of $\Omega$, with inclusion-preserving inverse.
\end{proposition}

\begin{definition}\label{def:spectrum}
The set $J(C)$ is the \emph{spectrum} of $C$. For $T\in\Z_n/\gen q$, the cyclic code with spectrum $T$ is denoted $M_T$.
\end{definition}

The set $J(C)$ is the classical set of \emph{nonzeros} of $C$: an exponent $a$ lies in $J(C)$ exactly when some codeword of $C$, read as a polynomial through Eq.~\eqref{eq:poly}, does not vanish at $\zeta^{a}$. Its complement $\Omega\setminus J(C)$ is the \emph{defining set} of $C$. The sign in $\chi_a(i)=\zeta^{-ai}$ is chosen so that both agree with the usual convention of the coding literature.

The $\gen q$-invariant subsets of $\Omega$ are exactly the unions of cyclotomic cosets, and the minimal nonempty ones are the single cosets. So Proposition~\ref{prop:spectrum} has the following consequence.

\begin{corollary}\label{cor:count}
There are exactly $2^{|\Z_n/\gen q|}$ cyclic codes of length $n$ over $\F_q$, and the minimal ones are the $M_T$ for $T\in\Z_n/\gen q$.
\end{corollary}

The spectrum is recorded through the isomorphism $a\mapsto\chi_a$ of $\Omega$ with its character group. Note that it depends on the choice of primitive root $\zeta$: a different primitive root replaces $J(C)$ by $uJ(C)$ for some $u\in\Z_n^\times$. All statements below are equivariant under this change.

\subsection{Idempotents}

\begin{lemma}[{\cite[Thm.~4.3.8]{HuffmanPless2003}}]\label{lem:idem}
Let $C$ be a cyclic code and let
\[
e_C:=\frac1n\sum_{a\in J(C)}\chi_a.
\]
Then, $e_C$ is an idempotent in $\F_q^\Omega$. Moreover, $C=e_C*\F_q^\Omega$ and $\F_q^\Omega=C\oplus C'$, where $C':=(\delta_0-e_C)*\F_q^\Omega$ with $J(C')=\Omega\setminus J(C)$.
\end{lemma}

\begin{lemma}[{\cite[Ch.~4]{HuffmanPless2003}}]\label{lem:sumspec}
If $C$ and $D$ are cyclic codes of length $n$ over $\F_q$, then $C+D$ is a cyclic code with $J(C+D)=J(C)\cup J(D)$. In particular, $\F_q^\Omega=\bigoplus_{T\in\Z_n/\gen q}M_T$, and for each $T$, the code $M_T':=\sum_{T'\ne T}M_{T'}$ satisfies $M_T'=(\delta_0-e_{M_T})*\F_q^\Omega$ and $\F_q^\Omega=M_T\oplus M_T'$.
\end{lemma}

\section{Linking codes and schemes}\label{sec:link}

In this section, we link the two theories reviewed in Sections~\ref{sec:cyc} and~\ref{sec:codes}. The link is a single space, namely the space of all functions on $\Z_n$ that are constant on the $q$-cyclotomic cosets. Each of the two theories equips this space with a basis, and the two bases are of entirely different origins.
\begin{enumerate}[label=\textup{(\alph*)}]
\item On the scheme side, it is isomorphic to the scheme algebra $\F_q\Cyc(\gen q,\Z_n)$, and the colors of the scheme give a basis of indicator functions.
\item On the coding side, it is spanned by the idempotent generators of the minimal cyclic codes.
\end{enumerate}

Theorem~\ref{thm:A} compares the two bases. We will show that a permutation of $\Z_n$ commutes with convolution by every element of the space if and only if it preserves the colors of the scheme. This condition may then be verified on either basis. On the first basis it states that the permutation is an automorphism of the scheme, and on the second that it fixes every cyclic code.

\subsection{The space \texorpdfstring{$\cS_q(n)$}{Sq(n)} and its two bases}\label{ssec:twobases}

Choose $R:=\Z_n$ and $K:=\gen q\le\Z_n^\times$ in the sense of Section~\ref{sec:cyc}, so that $\Aff_q(n)=\Aff_{\gen q}(\Z_n)$. The colors of $\Cyc(\gen q,\Z_n)$ are precisely the $q$-cyclotomic cosets of Section~\ref{ssec:spec}, so the rank of the scheme is $|\Z_n/\gen q|$. Recall from Theorem~\ref{thm:A} the abbreviation $\Aut_q(n)=\Aut\big(\Cyc(\gen q,\Z_n)\big)$, which contains $\Aff_q(n)$ by Proposition~\ref{prop:cycscheme}.

\begin{definition}\label{def:S}
Let $\cS_q(n)$ be the set of all $f\in\F_q^\Omega$ that are constant on each cyclotomic coset, or equivalently the set of all $f$ with $f(qi)=f(i)$ for every $i\in\Omega$.
\end{definition}

The first basis of $\cS_q(n)$ comes from the cyclotomic association scheme side.

\begin{proposition}\label{prop:S}
The set $\cS_q(n)$ is the subalgebra of $\F_q^\Omega$ fixed by $\mu_q$, and
\[
\F_q\Cyc(\gen q,\Z_n)\ \xrightarrow{\ \sim\ }\ \cS_q(n),\qquad T\longmapsto\one_T,
\]
is an isomorphism of $\F_q$-algebras.
\end{proposition}

\begin{proof}
The identity $\mu_qf=f$ says $f(q^{-1}i)=f(i)$ for all $i$, that is, $f$ is constant on the orbits of $\gen q$. So $\cS_q(n)$ is the fixed set of $\mu_q$, and it is a subalgebra of $\F_q^\Omega$ because $\mu_q$ acts as an algebra automorphism.

The colors of $\Cyc(\gen q,\Z_n)$ are the cyclotomic cosets. The scheme algebra $\F_q\Cyc(\gen q,\Z_n)$ defined in Definition~\ref{def:schemealg} is the free $\F_q$-vector space on the cyclotomic cosets, of dimension $\rk\big(\Cyc(\gen q,\Z_n)\big)$. The map $T\mapsto\one_T$ sends that basis to the indicators, which form a basis of $\cS_q(n)$.

It remains to show that this vector space isomorphism is an algebra homomorphism. For every $T, T' \in \Z_n / \gen q$ and $i \in \Z_n$,
\[
(\one_T*\one_{T'})(i)=\#\{(u,v)\in T\times T':u+v=i\} = c^{\,i\gen q}_{TT'},
\]
where the second equality is obtained by applying Axiom~\ref{ax:reg} to the pair $(0,i)$, with $r = T$, $s = T'$ and $t = i\gen q$. Hence $\one_T*\one_{T'}=\sum_{T''}c^{T''}_{TT'}\one_{T''}$, which matches the multiplication rule of the scheme algebra.
\end{proof}

The second basis of $\cS_q(n)$ comes from the coding side.

\begin{lemma}\label{lem:eCinS}
For every cyclic code $C$, $e_C\in\cS_q(n)$. Moreover, the $e_{M_T}$'s for $T\in\Z_n/\gen q$ are the primitive idempotents of $\F_q^\Omega$ and form a basis of $\cS_q(n)$.
\end{lemma}

\begin{proof}
By~\ref{ch:mu}, $\mu_qe_C=\frac1n\sum_{a\in J(C)}\chi_{q^{-1}a}=e_C$, since $q^{-1}J(C)=J(C)$, hence $e_C\in\cS_q(n)$ by Definition~\ref{def:S}. In particular, every $e_{M_T}$ lies in $\cS_q(n)$. The $e_{M_T}$'s are pairwise orthogonal nonzero idempotents summing to $\delta_0$, hence linearly independent. Moreover, there are $|\Z_n/\gen q|$ of them, which by Proposition~\ref{prop:S} is $\dim_{\F_q}\cS_q(n)$, so they form a basis of $\cS_q(n)$. They are primitive because the $M_T$ are the minimal cyclic codes, by Corollary~\ref{cor:count}.
\end{proof}

\subsection{Proof of Theorem~\ref{thm:A}}\label{ssec:proofA}

With both bases in hand, it remains to relate them. The following lemma describes the permutations that commute with convolution by invariant functions.

\begin{lemma}\label{lem:commute}
Let $\sigma\in\Sym(\Omega)$. Then, $\sigma \in \Aut_q(n)$ if and only if $\sigma(f*c)=f*\sigma(c)$ for all $f \in \cS_q(n)$ and $c \in \F_q^\Omega$.
\end{lemma}

\begin{proof}
Let $f,c\in\F_q^\Omega$. For $i\in\Omega$, evaluating at $\sigma i$, the function $\sigma(f*c)$ takes the value $\sum_jf(i-j)c(j)$, and $f*\sigma(c)$ takes the value $\sum_jf(\sigma i-j)c(\sigma^{-1}j) = \sum_jf(\sigma i-\sigma j)c(j)$. These agree for every $i$ and every $c$ if and only if
\begin{equation} \label{eq:comm}
f(\sigma i-\sigma j)=f(i-j)\qquad\text{for all }i,j\in\Omega.
\end{equation}

Suppose $\sigma\in\Aut_q(n)$, and let $f\in\cS_q(n)$. Then $f(i-j)$ depends only on the coset $(i-j)\gen q$, that is, on the color $\pi(j,i)$, which $\sigma$ preserves. So Eq.~\eqref{eq:comm} holds, and therefore $\sigma(f*c)=f*\sigma(c)$ for every $c$.

Conversely, suppose that $\sigma(f*c)=f*\sigma(c)$ for all $f\in\cS_q(n)$ and $c\in\F_q^\Omega$, so that Eq.~\eqref{eq:comm} holds for every $f\in\cS_q(n)$. Taking $f=\one_T$ gives $\one_T(\sigma i-\sigma j)=\one_T(i-j)$ for every $T\in\Z_n/\gen q$, that is, $(\sigma i-\sigma j)\gen q=(i-j)\gen q$. So $\sigma$ preserves every color and lies in $\Aut_q(n)$.
\end{proof}

Now, we are ready to prove Theorem~\ref{thm:A}.

\begin{proof}[Proof of Theorem~\ref{thm:A}]
Let $\sigma\in \Aut_q(n)$ and let $C$ be a cyclic code, with idempotent generator $e_C\in\F_q^\Omega$ as in Lemma~\ref{lem:idem}, so that $e_C*e_C=e_C$ and $C=e_C*\F_q^\Omega$. By Lemma~\ref{lem:eCinS}, $e_C\in\cS_q(n)$. So Lemma~\ref{lem:commute} gives $\sigma(e_C*c)=e_C*\sigma(c)$ for every $c\in\F_q^\Omega$. Since $\sigma$ is a bijection of $\F_q^\Omega$,
\[
\sigma C=\sigma\big(e_C*\F_q^\Omega\big)=e_C*\sigma\big(\F_q^\Omega\big)=e_C*\F_q^\Omega=C,
\]
hence $\sigma\in\PAut(C)$. This proves $\Aut_q(n)\le\bigcap_C\PAut(C)$.

Conversely, suppose that $\sigma\in\Sym(\Omega)$ fixes every minimal cyclic code, and fix $T\in\Z_n/\gen q$. Then, both $M_T$ and $M_T':=\sum_{T'\ne T}M_{T'}$ are $\sigma$-invariant, the latter because $\sigma$ acts linearly. By Lemma~\ref{lem:sumspec} and Lemma~\ref{lem:idem}, $\F_q^\Omega=M_T\oplus M_T'$ and convolution by $e_{M_T}$ is the projection onto $M_T$ along $M_T'$.

Let $c \in \F_q^\Omega$. By Lemma~\ref{lem:idem} the idempotent of $M_T'$ is $e_{M_T'}=\delta_0-e_{M_T}$, hence $c=e_{M_T}*c+e_{M_T'}*c$. Applying $\sigma$ puts $\sigma(e_{M_T}*c)$ in $M_T$ and $\sigma(e_{M_T'}*c)$ in $M_T'$. Since convolution by $e_{M_T}$ is the identity on $M_T$ and annihilates $M_T'$,
\[
e_{M_T}*\sigma(c)=e_{M_T}*\big(\sigma (e_{M_T}*c) + \sigma (e_{M_T'}* c)\big)=\sigma\big(e_{M_T}*c\big).
\]
Thus, $\sigma$ commutes with convolution by each $e_{M_T}$. The condition is linear in the convolving element, and by Lemma~\ref{lem:eCinS} the $e_{M_T}$'s form a basis of $\cS_q(n)$. So, $\sigma(f*c)=f*\sigma(c)$ for every $f\in\cS_q(n)$ and every $c$, hence $\sigma\in \Aut_q(n)$ by Lemma~\ref{lem:commute}. Therefore
\[
\Aut_q(n)\ \le\ \bigcap_C\PAut(C)\ \le\ \bigcap_{T\in\Z_n/\gen q}\PAut(M_T)\ \le\ \Aut_q(n),
\]
and the three groups coincide.
\end{proof}

\section{Exceptional lengths}\label{sec:examples}

This section gives an arithmetic characterization of exceptional lengths. Section~\ref{ssec:EP} reviews the structural characterization given by Evdokimov-Ponomarenko \cite{EvdokimovPonomarenko2003}. Section~\ref{ssec:proofB} proves Theorem~\ref{thm:B}. In each of Sections~\ref{ssec:famA}, \ref{ssec:famB} and~\ref{ssec:famC}, we give an infinite family of exceptional lengths, corresponding to the three conditions of Theorem~\ref{thm:B}. Finally, Section~\ref{ssec:prim} uses Theorem~\ref{thm:B} in the other direction, and shows that no length of the form $q^{m}-1$ is exceptional.

\subsection{The characterization of Evdokimov-Ponomarenko}\label{ssec:EP}

Proposition~\ref{prop:cycscheme} identifies $\Aff_K(R)$ as the affine part of $\Aut\big(\Cyc(K,R)\big)$, so the equality $\Aut\big(\Cyc(K,R)\big)=\Aff_K(R)$ holds if and only if $\Aut\big(\Cyc(K,R)\big)\le\AGL_1(R)$. That is, the automorphism group is as small as possible exactly when it consists of affine maps, and this is the condition that Evdokimov-Ponomarenko decide.

Two invariants are used in their structural characterization. The first is
\[
\Tr(n,q):=\{h\in\Z_n:\ \gen q+h=\gen q\}.
\]
It contains $0$ and is closed under addition, so it is a subgroup of $\Z_n$.

For the second, observe that a prime $p$ with $p\,\|\,n$ splits the unit group as $\Z_n^\times=\Z_p^\times\times\Z_{n/p}^\times$, by the Chinese remainder theorem. Let
\[
\Pr(n,q):=\{p\ \text{prime}:\ p\,\|\,n\ \text{and}\ \Z_p^\times\times\{1\}\le\gen q\}.
\]

\begin{theorem}[Evdokimov-Ponomarenko, {\cite[Thm.~6.1]{EvdokimovPonomarenko2003}}]\label{thm:EP}
Let $q$ be a prime power and $n\ge2$ with $\gcd(n,q)=1$. Then,
\[
\Aut_q(n)=\Aff_q(n)
\qquad\text{if and only if}\qquad
|\Tr(n,q)|\le2\ \text{ and }\ \Pr(n,q)\subseteq\{2,3\}.
\]
\end{theorem}

\subsection{Proof of Theorem~\ref{thm:B}}\label{ssec:proofB}

Theorem~\ref{thm:B} uses indices instead of the two invariants $\Tr(n,q)$ and $\Pr(n,q)$. The two lemmas below convert information on those invariants into equalities of the index $\idx_n(q)$, which is the index of $\gen q$ in $\Z_n^\times$. Both pass through the kernels of the reduction: for a divisor $f$ of $n$, let
\[
	K_{n, f} := \ker (\Z_n^\times \twoheadrightarrow \Z_f^\times),
\]
which is a subgroup of order $\varphi(n)/\varphi(f)$.

\begin{lemma}\label{lem:idxmono}
Let $f$ be a divisor of $n$. Then, $K_{n,f}\subseteq\gen q$ if and only if $\idx_f(q)=\idx_n(q)$.
\end{lemma}

\begin{proof}
The reduction carries $\gen q$ onto the subgroup of $\Z_f^\times$ generated by $q$, which has order $\ord_f(q)$, so $\gen q\cap K_{n,f}$ has order $\ord_n(q)/\ord_f(q)$. Dividing the order of $K_{n,f}$ by that order gives
\[
\big[K_{n,f}:\gen q\cap K_{n,f}\big]=\frac{\varphi(n)/\varphi(f)}{\ord_n(q)/\ord_f(q)}=\frac{\idx_n(q)}{\idx_f(q)},
\]
from which the result follows.
\end{proof}

\begin{lemma}\label{lem:stcrit}
Let $d\mid n$ and let $U=(n/d)\Z_n$ be the subgroup of $\Z_n$ of order $d$. Then, $U\le\Tr(n,q)$ if and only if $\idx_{n/d}(q)=\idx_n(q)$ and every prime dividing $d$ also divides $n/d$.
\end{lemma}

\begin{proof}
We use the abbreviation $K:=K_{n,n/d}$ for simplicity. The coset $1+U$ consists of the $d$ residues congruent to $1$ modulo $n/d$, and $K$ consists of those of them that are units. So $K\subseteq1+U$, and the two coincide precisely when $|K|=d$. Here $|K|=\varphi(n)/\varphi(n/d)$, which equals $d$ times the product of $1-1/p$ over the primes $p$ dividing $n$ but not $n/d$. Therefore, $K=1+U$ if and only if every prime dividing $d$ also divides $n/d$.

We first prove the sufficiency. By Lemma~\ref{lem:idxmono}, $K\subseteq\gen q$, hence $1 + U = K \subseteq\gen q$. Since $U$ is an ideal of $\Z_n$, we have $\gen q\,U\subseteq U$, so $\gen q+U\subseteq\gen q(1+\gen qU)\subseteq\gen q(1+U)\subseteq\gen q$, and equality holds as both sides have the same size. Therefore, $U \le \Tr(n,q)$.

Assume conversely that $U\le\Tr(n,q)$. Taking the element $1\in\gen q$ in the identity $\gen q+u=\gen q$ gives $1+U\subseteq\gen q\subseteq\Z_n^\times$, so every member of $1+U$ is a unit. Therefore $K=1+U$, which implies that every prime dividing $d$ also divides $n/d$. Moreover, since $K$ lies inside $\gen q$, Lemma~\ref{lem:idxmono} gives $\idx_{n/d}(q)=\idx_n(q)$.
\end{proof}

Now, we have all the ingredients to prove the arithmetic characterization.

\begin{proof}[Proof of Theorem~\ref{thm:B}]
The result is trivial for $n = 1$, and we assume that $n \geq 2$. By Theorem~\ref{thm:EP}, $\Aut_q(n)\ne\Aff_q(n)$ if and only if either $\Pr(n,q)$ contains a prime $p\ge5$ or $|\Tr(n,q)|>2$. We show that the first is equivalent to Condition~\ref{it:E1}, and that the second holds if and only if Condition~\ref{it:E2} or Condition~\ref{it:E3} does.

Let $p\,\|\,n$. The Chinese remainder theorem gives $K_{n,n/p}=\Z_p^\times\times\{1\}$, so $p\in\Pr(n,q)$ if and only if $K_{n,n/p}\subseteq\gen q$, which by Lemma~\ref{lem:idxmono} says that $\idx_{n/p}(q)=\idx_n(q)$. Restricting to $p\ge5$ identifies the fact that $\Pr(n,q)$ contains a prime $p\ge5$ with Condition~\ref{it:E1}.

Suppose now that Condition~\ref{it:E2} holds at the odd prime $p$. Applying Lemma~\ref{lem:stcrit} with $d = p$, we obtain a subgroup of order $p$ inside $\Tr(n,q)$, and $|\Tr(n,q)|\ge p>2$. Condition~\ref{it:E3} runs the same way with $d = 4$, for which Lemma~\ref{lem:stcrit} gives $|\Tr(n,q)|\ge 4>2$.

Conversely suppose $|\Tr(n,q)|>2$, and let $t:=|\Tr(n,q)|$. Being a subgroup of the cyclic group $\Z_n$, the group $\Tr(n,q)$ equals $(n/t)\Z_n$, and it contains $(n/d)\Z_n$ for every divisor $d$ of $t$. Let $d$ be an odd prime divisor of $t$ if one exists, and $d=4$ otherwise, which is available because a power of $2$ exceeding $2$ is divisible by $4$. Then, Lemma~\ref{lem:stcrit} gives the equality $\idx_{n/d}(q)=\idx_n(q)$, and the fact that every prime dividing $d$ also divides $n/d$. If $d$ is an odd prime, the latter reads $d\mid n/d$, that is, $d^{2}\mid n$, which is Condition~\ref{it:E2}; and if $d=4$, it reads $2\mid n/4$, that is, $8\mid n$, which is Condition~\ref{it:E3}.
\end{proof}

\subsection{A family from a primitive root}\label{ssec:famA}

Condition~\ref{it:E1} asks the index identity to hold at a prime $p\ge5$ with $p\,\|\,n$. In that case the identity has a simple characterization in terms of orders.

\begin{proposition}\label{prop:famA}
Let $p$ be a prime with $p\,\|\,n$ and let $m:=n/p$. Then,
\[
\idx_m(q)=\idx_n(q)\iff \ord_p(q)=p-1\ \text{ and }\ \gcd\big(p-1,\ \ord_m(q)\big)=1 .
\]
\end{proposition}

\begin{proof}
Since $p\,\|\,n$, we have $\varphi(n)=(p-1)\varphi(m)$, so the identity $\idx_m(q)=\idx_n(q)$ says that $\ord_n(q)=(p-1)\ord_m(q)$. The Chinese remainder theorem gives $\ord_n(q)=\operatorname{lcm}\big(\ord_p(q),\ord_m(q)\big)$. Writing $t:=\ord_p(q)$ and $u:=\ord_m(q)$, the identity $\operatorname{lcm}(t,u)=(p-1)u$ reads $t/\gcd(t,u)=p-1$, and since $t$ divides $p-1$ this holds if and only if $t=p-1$ and $\gcd(t,u)=1$.
\end{proof}

\begin{example}\label{ex:famA}
Take $q=2$ and $p=5$, where $2$ is a primitive root modulo $5$. For odd $k\ge3$, let $m:=2^{k}-1$. Note that $\ord_m(2)=k$ is odd and hence coprime to $p-1=4$, and $5\nmid m$, since $5\mid2^{k}-1$ would force $4\mid k$. Setting $n:=5m$, so that $5\,\|\,n$, Proposition~\ref{prop:famA} gives $\idx_m(2)=\idx_n(2)$, which is Condition~\ref{it:E1} at $p=5$. So by Theorem~\ref{thm:B} every $n=5\,(2^{k}-1)$ with odd $k\ge3$ is exceptional for $q=2$.
\end{example}

\subsection{A family from a prime power}\label{ssec:famB}

Condition~\ref{it:E2} asks the index identity to hold at a prime $p$ with $p^{2}\mid n$. The simplest lengths of that kind are the prime powers themselves.

\begin{proposition}\label{prop:famB}
Let $p$ be an odd prime not dividing $q$ and suppose
\[
q^{\,\ord_p(q)}\not\equiv1\pmod{p^{2}}.
\]
Then, for $j \ge 1$,
\[
\idx_{p^{j}}(q)=\frac{p-1}{\ord_p(q)}.
\]
\end{proposition}

\begin{proof}
Let $t:=\ord_p(q)$. The hypothesis says that $q^{t}\ne1$ in $\Z_{p^{2}}^\times$, so the order of $q$ modulo $p^{2}$ is $pt$, and the standard description of $\Z_{p^{j}}^\times$ for odd $p$ then gives $\ord_{p^{j}}(q)=p^{\,j-1}t$ for every $j\ge1$. Dividing $\varphi(p^{j})=p^{\,j-1}(p-1)$ by this order gives $\idx_{p^{j}}(q)=(p-1)/\ord_p(q)$.
\end{proof}

\begin{example}\label{ex:famB}
For $q=2$ the hypothesis of Proposition~\ref{prop:famB} holds at $p=3$, where $\ord_3(2)=2$ and $2^{2}=4\not\equiv1\pmod 9$; at $p=5$, where $\ord_5(2)=4$ and $2^{4}=16\not\equiv1\pmod{25}$; and at $p=7$, where $\ord_7(2)=3$ and $2^{3}=8\not\equiv1\pmod{49}$. For $n=p^{k}$ with $k\ge2$ the proposition gives $\idx_{n/p}(2)=\idx_n(2)$, which together with $p^{2}\mid n$ is Condition~\ref{it:E2}. So by Theorem~\ref{thm:B}, the lengths $3^{k}$, $5^{k}$ and $7^{k}$ with $k\ge2$ are all exceptional for $q=2$.
\end{example}

\subsection{A family from a power of two}\label{ssec:famC}

Only Condition~\ref{it:E3} can make a power of $2$ exceptional, since the other two ask for an odd prime divisor of $n$. At $p=2$ the index $\idx_{2^{j}}(q)$ is eventually constant in $j$, as Proposition~\ref{prop:famC} below makes precise. For a length $n=2^{k}$, Condition~\ref{it:E3} asks that it have become constant by the time $j$ reaches $k-2$.

The exponent at which the index becomes constant is governed by a single invariant of $q$. For an odd prime power $q$, let $e:=\ord_4(q)$ and let
\[
s(q):=v_2\big(q^{\,e}-1\big) = \begin{cases}
v_2(q-1), & q\equiv1\pmod4, \\
v_2(q^2-1), & q\equiv3\pmod4.
\end{cases}
\]
In either case $s(q)\ge2$. Equivalently, $s(q)$ is the largest $j\ge2$ with $\ord_{2^{j}}(q)=\ord_4(q)$.

\begin{proposition}\label{prop:famC}
Let $q$ be an odd prime power, and let $e:=\ord_4(q)$ and $s:=s(q)$. Then, for $j \ge 2$,
\[
\idx_{2^{j}}(q)=\frac{2^{\,\min(j,\,s)-1}}{e}.
\]
\end{proposition}

\begin{proof}
Since $2^{s}$ exactly divides $q^{e}-1$ and $s\ge2$, the integer $q^{e}+1$ is congruent to $2$ modulo $4$, and for $t\ge1$ the integer $q^{\,e2^{t}}+1$ is an odd square plus one, hence congruent to $2$ modulo $8$. Factoring
\[
q^{\,e2^{i}}-1=\big(q^{e}-1\big)\prod_{t=0}^{i-1}\big(q^{\,e2^{t}}+1\big)
\]
and adding the valuations therefore gives $v_2\big(q^{\,e2^{i}}-1\big)=s+i$ for every $i\ge0$. Now let $j\ge2$. The order $\ord_{2^{j}}(q)$ divides $\varphi(2^{j})=2^{\,j-1}$ and is a multiple of $e$, so it equals $e2^{i}$ for some $i\ge0$, and it is the least such value for which $2^{j}$ divides $q^{\,e2^{i}}-1$, that is, for which $j\le s+i$. Hence for $j\ge2$,
\[
\ord_{2^{j}}(q)=e\,2^{\max(0,\,j-s)}.
\]
Dividing $\varphi(2^{j})=2^{\,j-1}$ by this order gives $\idx_{2^{j}}(q)=2^{\,\min(j,\,s)-1}/e$.
\end{proof}

\begin{example}\label{ex:famC}
For $n=2^{k}$ Conditions~\ref{it:E1} and~\ref{it:E2} are vacuous, so such a length is exceptional exactly when Condition~\ref{it:E3} holds, that is, when $k\ge3$ and $\idx_{2^{k-2}}(q)=\idx_{2^{k}}(q)$. For $k\ge4$ both exponents are at least $2$, so Proposition~\ref{prop:famC} turns that identity into $\min(k-2,s(q))=\min(k,s(q))$, which holds precisely when $k\ge s(q)+2$; and $k=3$ is excluded, since $\idx_{2}(q)=1$ while $\idx_{8}(q)=4/\ord_8(q)\ge2$ because $\Z_8^\times$ has exponent $2$. So $n=2^{k}$ is exceptional exactly when $k\ge s(q)+2$. Take $q=5$, so that $q\equiv1\pmod 4$ and $s(5)=v_2(4)=2$; then every $n=2^{k}$ with $k\ge4$ is exceptional for $q=5$.
\end{example}

\subsection{Lengths of the form \texorpdfstring{$q^{m}-1$}{q^m-1}}\label{ssec:prim}

Every condition of Theorem~\ref{thm:B} forces the order of $q$ to drop by a factor of at least three when a divisor is removed from the length. Modulo $q^{m}-1$, the order of $q$ is $m$, the smallest that a length of this size permits. All three conditions therefore fail for the length $n=q^{m}-1$.

\begin{proposition}\label{prop:prim}
For every $m\ge1$, the length $n=q^{m}-1$ is not exceptional for $q$.
\end{proposition}

\begin{proof}
Suppose that $n=q^{m}-1$ is exceptional. Since $q^{m}\equiv1\pmod n$, and since $0<q^{k}-1<n$ whenever $1\le k<m$, we have $\ord_n(q)=m$. By Theorem~\ref{thm:B} one of the three conditions holds, and each of them has the form $\idx_{n/d}(q)=\idx_n(q)$ for a divisor $d$ of $n$. Equivalently, $m = \lambda t$, where $t := \ord_{n/d}(q)$ and $\lambda := \frac{\varphi(n)}{\varphi(n/d)}$.

In Condition~\ref{it:E1}, $(d,\lambda)=(p,p-1)$ with $p\ge5$. In Condition~\ref{it:E2}, $(d,\lambda)=(p,p)$ with $p\ge3$. In Condition~\ref{it:E3}, $(d,\lambda)=(4,4)$. In all three cases, $\lambda \ge 3$ and $d \le \lambda + 1$. Since $\ord_{n/d}(q)=t$, the divisor $n/d$ divides $q^{t}-1$. Then,
\[
q^{\lambda t}-1=n\le d\,\big(q^{t}-1\big)<(\lambda+1)\,q^{t}-1.
\]
Adding $1$ and dividing by $q^{t}$ gives $2^{\lambda - 1} \leq q^{(\lambda - 1) t} < \lambda+1$. This is a contradiction, since $2^{\lambda-1}\ge\lambda+1$ for every $\lambda\ge3$.
\end{proof}

\section{The density of exceptional lengths}\label{sec:density}

Section~\ref{sec:examples} produces infinitely many exceptional lengths. In this section we show that they are nonetheless rare, and prove Theorem~\ref{thm:C}.

The three conditions of Theorem~\ref{thm:B} are handled in two groups, according to how the prime they name divides $n$. Condition~\ref{it:E1} names a prime dividing $n$ exactly once, while Conditions~\ref{it:E2} and~\ref{it:E3} name one dividing $n$ at least twice. Accordingly, let
\begin{align*}
\cP(q)&:=\{n \in \N :\ \gcd(n,q)=1\ \text{and $n$ satisfies Condition~\ref{it:E1}}\},\\
\cQ(q)&:=\{n \in \N :\ \gcd(n,q)=1\ \text{and $n$ satisfies Condition~\ref{it:E2} or~\ref{it:E3}}\},
\end{align*}
where the letters stand for {\em primitive root} and for {\em square}. By Theorem~\ref{thm:B}, a length coprime to $q$ is exceptional if and only if it lies in $\cP(q)\cup\cQ(q)$. So Theorem~\ref{thm:C} is equivalent to the following two statements.

\begin{proposition}\label{prop:densP}
The set $\cP(q)$ has density $0$.
\end{proposition}

\begin{proposition}\label{prop:densQ}
The set $\cQ(q)$ has density $0$.
\end{proposition}

Both proofs run in two steps. The first shows that $n$ has at most one prime factor $\ell$ for which $\ord_\ell(q)$ is divisible by a certain prime power $r$, and the second shows that the integers with that property have density $0$.

The second step is proved in Section~\ref{ssec:analytic}. Section~\ref{ssec:densP} carries out the reduction for $\cP(q)$ and proves Proposition~\ref{prop:densP} with the prime power $r=2$. Section~\ref{ssec:densQ} carries out the reduction for $\cQ(q)$ and proves Proposition~\ref{prop:densQ} with prime powers $r$ that grow with $v_p(n)$.

\subsection{The counting lemma}\label{ssec:analytic}

The second step of both proofs is Lemma~\ref{lem:key} below. It rests on one theorem from analytic number theory.

\begin{theorem}[Wiertelak \cite{Wiertelak1978,Wiertelak1978a,Wiertelak2000}; see also \cite{Moree2005}]\label{thm:orddens}
For every integer $r\ge2$, the set of primes $\ell$ with $r\mid\ord_\ell(q)$ has positive density in the set of primes.
\end{theorem}

\begin{lemma}\label{lem:key}
Let $r\ge2$ be an integer. Then,
\[
\{n \in \N:\ n\ \text{has at most one prime factor}\ \ell\ \text{with}\ r\mid\ord_\ell(q)\}
\]
has density $0$.
\end{lemma}

\begin{proof}
Let $L:=\{\ell\ \text{prime}:\ r\mid\ord_\ell(q)\}$, and let $\ell_1<\ell_2<\dots$ be its elements. By Theorem~\ref{thm:orddens}, the set $L$ has positive density in the set of primes, so $\#\{i:\ \ell_i\le t\}\gg t/\log t$, and partial summation gives $\sum_{i=1}^\infty 1/\ell_i=\infty$.

Consider $k\in \N$ and let
\[
P_k:=\prod_{i=1}^k\Big(1-\frac1{\ell_i}\Big),\qquad S_k:=\sum_{i=1}^k\frac1{\ell_i}.
\]
Those $n$ divisible by none of the $\ell_i$ have density $P_k$, and for each $i$ those divisible by $\ell_i$ and by no other have density $P_k/(\ell_i-1)$. An integer with at most one prime factor in $L$ has at most one among $\ell_1,\dots,\ell_k$, so the set in the lemma statement has upper density at most
\[
P_k\Big(1+\sum_{i\le k}\frac1{\ell_i-1}\Big) \leq e^{-S_k} (1 + 2 S_k).
\]
The result follows from taking the limit $k \to \infty$.
\end{proof}

\subsection{The primitive root condition}\label{ssec:densP}

Because Condition~\ref{it:E1} involves only primes $p\ge5$, Proposition~\ref{prop:densP} reduces to the case $r=2$ of Lemma~\ref{lem:key}.

\begin{proof}[Proof of Proposition~\ref{prop:densP}]
Let $n\in\cP(q)$, so that $\idx_{n/p}(q)=\idx_n(q)$ for some prime $p\ge5$ with $p\,\|\,n$, and let $m:=n/p$. By Proposition~\ref{prop:famA}, we have $\gcd(p-1,\ord_m(q))=1$. Since $p-1$ is even, $\ord_m(q)$ is odd. For every prime $\ell\mid m$ the reduction $\Z_m^\times\twoheadrightarrow\Z_\ell^\times$ gives $\ord_\ell(q)\mid\ord_m(q)$, hence $\ord_\ell(q)$ is odd as well. Therefore, $p$ is the only prime factor $\ell$ of $n$ with $2\mid\ord_\ell(q)$. Lemma~\ref{lem:key} with $r=2$ gives the desired result.
\end{proof}

\subsection{The square condition}\label{ssec:densQ}

Proposition~\ref{prop:densQ} requires prime powers $r$ that vary with $n$. The following lemma supplies them.

\begin{lemma}\label{lem:local}
Let $p$ be a prime, $a := v_p(n)$ and $m:=n/p^{a}$, and let $c$ be an integer with $1 \leq c < a$. If $\ord_n(q)=p^{c}\,\ord_{n/p^{c}}(q)$, then
\[
v_p\big(\ord_m(q)\big)\ \le\ a-1-c.
\]
\end{lemma}

\begin{proof}
By the Chinese remainder theorem,
\[
\ord_n(q)=\operatorname{lcm}\big(\ord_{p^{a}}(q),\ord_m(q)\big),\quad
\ord_{n/p^{c}}(q)=\operatorname{lcm}\big(\ord_{p^{a-c}}(q),\ord_m(q)\big).
\]
Each $\ord_{p^{j}}(q)$ divides $\varphi(p^{j})=p^{j-1}(p-1)$ and has the same prime-to-$p$ part as $\ord_p(q)$, so the two orders above have equal prime-to-$p$ parts. Write $\alpha_j:=v_p(\ord_{p^{j}}(q))$ and $b:=v_p(\ord_m(q))$. Their $p$-parts are then $\max(\alpha_a,b)$ and $\max(\alpha_{a-c},b)$, so
\[
\max(\alpha_a,b)=\max(\alpha_{a-c},b)+c.
\]
The right-hand side exceeds $b$, so the left-hand side is $\alpha_a$, and then $\alpha_a\ge b+c$. Finally $\alpha_a\le a-1$, because $\ord_{p^{a}}(q)$ divides $p^{a-1}(p-1)$, so $b\le a-1-c$.
\end{proof}

\begin{proof}[Proof of Proposition~\ref{prop:densQ}]
Let $n\in\cQ(q)$. By definition there is a prime $p$ for which either $p$ is odd with $p^{2}\mid n$ and $\idx_{n/p}(q)=\idx_n(q)$, or $p=2$ with $8\mid n$ and $\idx_{n/4}(q)=\idx_n(q)$. Let $c:=1$ in the first case and $c:=2$ in the second, $a:=v_p(n)$ and $m:=n/p^{a}$. In both cases, every prime dividing $p^{c}$ divides $n/p^{c}$, so $\varphi(n)/\varphi(n/p^{c})=p^{c}$ and the index identity reads
\begin{equation} \label{eq:ord}
\ord_n(q)=p^{c}\,\ord_{n/p^{c}}(q).
\end{equation}
Moreover $c<a$ in both cases, since $p^{2}\mid n$ reads $2\le a$ in the first and $8\mid n$ reads $3\le a$ in the second.

For a prime $p$ and an integer $a \ge 2$, let $\cQ_{(p, a)}(q)$ be the set of $n \in \cQ(q)$ such that Eq.~\eqref{eq:ord} holds with $p$ satisfying $v_p(n)=a$. By Lemma~\ref{lem:local}, we have $v_p(\ord_m(q))\le a-1-c$, so $p^{\,a-c}\nmid\ord_\ell(q)$ for every prime $\ell\mid m$, since $\ord_\ell(q)$ divides $\ord_m(q)$. Thus $p$ is the only prime factor $\ell$ of $n$ that can satisfy $p^{\,a-c}\mid\ord_\ell(q)$, and Lemma~\ref{lem:key} with $r=p^{\,a-c}$ shows that $\cQ_{(p, a)}(q)$ has density $0$.

Let $y\ge2$, and let $\cQ^{1}(q)$ be the subset of $\cQ(q)$ consisting of those $n$ that are divisible by some $p^{a} > y$ with prime $p$ and $a\ge2$, and let $\cQ^{0}(q) := \cQ(q) \setminus \cQ^{1}(q)$. Note that $\cQ^0(q)$ is a subset of the finite union $\bigcup_{p^a \leq y} \cQ_{(p, a)}(q)$, hence has density $0$. For $\cQ^1(q)$,
\[
\#\{n \leq x : n \in \cQ^{1}(q)\} \leq x\sum_{\substack{p^{a}>y\\ a\ge2}}\frac1{p^{a}} \ll \frac{x}{\sqrt{y}}.
\]
Therefore, $\cQ(q)$ has upper density $\ll y^{-1/2}$. Since $y\ge2$ was arbitrary, that upper density is $0$.
\end{proof}

\section*{Acknowledgments}

This research was supported by grants from the NSFC (12350710787) and NSFC (12471311). The authors used Claude Opus 5, under their guidance, to (1) perform calculations, (2) assist in proofreading the manuscript, and (3) improve the English writing.

\bibliographystyle{plainnat}
\bibliography{part1}

@Article{BabaiGodsil1982,
  author   = {Babai, Laszlo and Godsil, Chris D.},
  journal  = {Eur. J. Comb.},
  title    = {On the automorphism groups of almost all {Cayley} graphs},
  year     = {1982},
  pages    = {9--15},
  volume   = {3},
  doi      = {10.1016/S0195-6698(82)80003-6},
  fjournal = {European Journal of Combinatorics},
  language = {English},
  zbl      = {0483.05033},
  zbmath   = {3758368},
}

@Book{BannaiIto1984,
  author    = {Bannai, Eiichi and Ito, Tatsuro},
  publisher = {The Benjamin/Cummings Publishing Company, Reading, MA},
  title     = {Algebraic combinatorics. {I}: {Association} schemes},
  year      = {1984},
  series    = {Math. Lect. Note Ser.},
  fseries   = {Mathematics Lecture Note Series},
  language  = {English},
  zbl       = {0555.05019},
  zbmath    = {3884178},
}

@Article{BergerCharpin1996,
  author   = {Berger, Thierry P. and Charpin, Pascale},
  journal  = {IEEE Trans. Inf. Theory},
  title    = {The permutation group of affine-invariant extended cyclic codes},
  year     = {1996},
  number   = {6},
  pages    = {2194--2209},
  volume   = {42},
  doi      = {10.1109/18.556607},
  fjournal = {IEEE Transactions on Information Theory},
  language = {English},
  zbl      = {0883.94012},
  zbmath   = {1008638},
}

@Article{Berman1967,
  author       = {Berman, S. D.},
  journal      = {Cybernetics},
  title        = {Semisimple cyclic and abelian codes. {II}},
  year         = {1967},
  number       = {3},
  pages        = {17--23},
  volume       = {3},
  doi          = {10.1007/BF01119999},
  language     = {English},
  zbl          = {1026.94551},
  zbmath       = {1251956},
}

@Article{BienertKlopsch2010,
  author   = {Bienert, Rolf and Klopsch, Benjamin},
  journal  = {J. Algebr. Comb.},
  title    = {Automorphism groups of cyclic codes},
  year     = {2010},
  number   = {1},
  pages    = {33--52},
  volume   = {31},
  doi      = {10.1007/s10801-009-0179-y},
  fjournal = {Journal of Algebraic Combinatorics},
  language = {English},
  zbl      = {1195.94083},
  zbmath   = {5682322},
}

@Article{BoseMesner1959,
  author   = {Bose, R. C. and Mesner, Dale M.},
  journal  = {Ann. Math. Stat.},
  title    = {On linear associative algebras corresponding to association schemes of partially balanced designs},
  year     = {1959},
  pages    = {21--38},
  volume   = {30},
  doi      = {10.1214/aoms/1177706356},
  fjournal = {Annals of Mathematical Statistics},
  language = {English},
  zbl      = {0089.15002},
  zbmath   = {3145663},
}

@Article{BoseShimamoto1952,
  author   = {Bose, R. C. and Shimamoto, T.},
  journal  = {J. Am. Stat. Assoc.},
  title    = {Classification and analysis of partially balanced incomplete block designs with two associate classes},
  year     = {1952},
  pages    = {151--184},
  volume   = {47},
  doi      = {10.2307/2280741},
  fjournal = {Journal of the American Statistical Association},
  language = {English},
  zbl      = {0048.11603},
  zbmath   = {3075255},
}

@InCollection{Charpin1998,
  author    = {Charpin, Pascale},
  booktitle = {Handbook of coding theory. Vol. 1. Part 1: Algebraic coding. Vol. 2. Part 2: Connections, Part 3: Applications},
  publisher = {Amsterdam: Elsevier},
  title     = {Open problems on cyclic codes},
  year      = {1998},
  pages     = {963--1063},
  language  = {English},
  zbl       = {0927.94017},
  zbmath    = {1284424},
}

@Article{Delsarte1973,
  author    = {Delsarte, P.},
  journal   = {Philips Res. Rep. Suppl.},
  title     = {An algebraic approach to the association schemes of coding theory},
  year      = {1973},
  volume    = {10},
  fjournal  = {Philips Research Reports Supplements},
  language  = {English},
  zbl       = {1075.05606},
  zbmath    = {2232233},
}

@Article{EvdokimovPonomarenko2003,
  author   = {Evdokimov, S. A. and Ponomarenko, I. N.},
  journal  = {St. Petersbg. Math. J.},
  title    = {Characterization of cyclotomic schemes and normal {Schur} rings over a cyclic group},
  year     = {2003},
  number   = {2},
  pages    = {11--55},
  volume   = {14},
  fjournal = {St. Petersburg Mathematical Journal},
  language = {English},
  zbl      = {1056.20001},
  zbmath   = {2007658},
}

@Article{EvdokimovPonomarenko2008,
  author   = {Evdokimov, S. A. and Ponomarenko, I. N.},
  journal  = {St. Petersbg. Math. J.},
  title    = {Normal cyclotomic schemes over a finite commutative ring},
  year     = {2008},
  number   = {6},
  pages    = {911--929},
  volume   = {19},
  doi      = {10.1090/S1061-0022-08-01027-3},
  fjournal = {St. Petersburg Mathematical Journal},
  language = {English},
  zbl      = {1206.13029},
  zbmath   = {5859437},
}

@Misc{Feng2026,
  author       = {Tao Feng and Henk D. L. Hollmann and Weicong Li and Qing Xiang},
  howpublished = {Preprint, {arXiv}:2603.01904 [math.{CO}]},
  title        = {The permutation automorphism groups of irreducible cyclic codes},
  year         = {2026},
  arxiv        = {arXiv:2603.01904},
  url          = {https://arxiv.org/abs/2603.01904},
}

@Article{GoldbachClaasen1992,
  author   = {Goldbach, R. W. and Claasen, H. L.},
  journal  = {Indag. Math., New Ser.},
  title    = {Cyclotomic schemes over finite rings},
  year     = {1992},
  number   = {3},
  pages    = {301--312},
  volume   = {3},
  doi      = {10.1016/0019-3577(92)90037-L},
  fjournal = {Indagationes Mathematicae. New Series},
  language = {English},
  zbl      = {0780.05059},
  zbmath   = {469238},
}

@Article{GoldbachClaasen1992a,
  author   = {Goldbach, R. W. and Claasen, H. L.},
  journal  = {Indag. Math., New Ser.},
  title    = {Cyclotomic schemes over finite, commutative, admissible rings},
  year     = {1992},
  number   = {3},
  pages    = {277--299},
  volume   = {3},
  doi      = {10.1016/0019-3577(92)90036-K},
  fjournal = {Indagationes Mathematicae. New Series},
  language = {English},
  zbl      = {0780.05058},
  zbmath   = {469237},
}

@Article{GuendaGulliver2013,
  author   = {Guenda, Kenza and Gulliver, T. Aaron},
  journal  = {J. Algebr. Comb.},
  title    = {On the permutation groups of cyclic codes},
  year     = {2013},
  number   = {1},
  pages    = {197--208},
  volume   = {38},
  doi      = {10.1007/s10801-012-0399-4},
  fjournal = {Journal of Algebraic Combinatorics},
  language = {English},
  zbl      = {1273.94405},
  zbmath   = {6192144},
}

@Article{Higman1971,
  author   = {Higman, D. G.},
  journal  = {Rend. Semin. Mat. Univ. Padova},
  title    = {Coherent configurations. {I}},
  year     = {1971},
  pages    = {1--25},
  volume   = {44},
  fjournal = {Rendiconti del Seminario Matematico della Universit{\`a} di Padova},
  language = {English},
  url      = {https://eudml.org/doc/107354},
  zbl      = {0279.05025},
  zbmath   = {3438886},
}

@Article{HuffmanJobPless1993,
  author   = {Huffman, W. Cary and Job, Vanessa and Pless, Vera S.},
  journal  = {J. Comb. Theory, Ser. A},
  title    = {Multipliers and generalized multipliers of cyclic objects and cyclic codes},
  year     = {1993},
  number   = {2},
  pages    = {183--215},
  volume   = {62},
  doi      = {10.1016/0097-3165(93)90043-8},
  fjournal = {Journal of Combinatorial Theory. Series A},
  language = {English},
  zbl      = {0772.94011},
  zbmath   = {221983},
}

@Book{HuffmanPless2003,
  author    = {Huffman, W. Cary and Pless, Vera},
  publisher = {Cambridge: Cambridge University Press},
  title     = {Fundamentals of error-correcting codes},
  year      = {2003},
  doi       = {10.1017/CBO9780511807077},
  language  = {English},
  zbl       = {1099.94030},
  zbmath    = {2006094},
}

@Article{LeungMa1990,
  author   = {Leung, Ka Hin and Ma, Siu Lun},
  journal  = {J. Pure Appl. Algebra},
  title    = {The structure of {Schur} rings over cyclic groups},
  year     = {1990},
  number   = {3},
  pages    = {287--302},
  volume   = {66},
  doi      = {10.1016/0022-4049(90)90032-D},
  fjournal = {Journal of Pure and Applied Algebra},
  language = {English},
  zbl      = {0761.20001},
  zbmath   = {5595},
}

@Article{LeungMan1996,
  author   = {Leung, Ka Hin and Man, Shing Hing},
  journal  = {J. Algebra},
  title    = {On {Schur} rings over cyclic groups. {II}},
  year     = {1996},
  number   = {2},
  pages    = {273--285},
  volume   = {183},
  doi      = {10.1006/jabr.1996.0220},
  fjournal = {Journal of Algebra},
  language = {English},
  zbl      = {0860.20005},
  zbmath   = {923685},
}

@Article{LeungMan1998,
  author   = {Leung, Ka Hin and Man, Shing Hing},
  journal  = {Isr. J. Math.},
  title    = {On {Schur} rings over cyclic groups},
  year     = {1998},
  pages    = {251--267},
  volume   = {106},
  doi      = {10.1007/BF02773471},
  fjournal = {Israel Journal of Mathematics},
  language = {English},
  zbl      = {0962.20002},
  zbmath   = {1216496},
}

@Article{MacWilliams1970,
  author   = {MacWilliams, F. J.},
  journal  = {Bell Syst. Tech. J.},
  title    = {Binary codes which are ideals in the group algebra of an {Abelian} group},
  year     = {1970},
  pages    = {987--1011},
  volume   = {49},
  doi      = {10.1002/j.1538-7305.1970.tb01812.x},
  fjournal = {Bell System Technical Journal},
  language = {English},
  zbl      = {0205.20501},
  zbmath   = {3326145},
}

@Book{MacWilliamsSloane1977,
  author    = {MacWilliams, F. J. and Sloane, N. J. A.},
  publisher = {Elsevier (North-Holland), Amsterdam},
  title     = {The theory of error-correcting codes. {Parts} {I}, {II}},
  year      = {1977},
  series    = {North-Holland Math. Libr.},
  volume    = {16},
  fseries   = {North-Holland Mathematical Library},
  language  = {English},
  zbl       = {0369.94008},
  zbmath    = {3577144},
}

@Article{Moree2005,
  author   = {Moree, Pieter},
  journal  = {Funct. Approximatio, Comment. Math.},
  title    = {On primes {{\(p\)}} for which {{\(d\)}} divides {{\(\text{ord}_p(g)\)}}},
  year     = {2005},
  pages    = {85--95},
  volume   = {33},
  fjournal = {Functiones et Approximatio. Commentarii Mathematici},
  language = {English},
  zbl      = {1228.11152},
  zbmath   = {5135205},
}

@Article{Muzychuk1994,
  author   = {Muzychuk, Mikhail E.},
  journal  = {J. Algebra},
  title    = {On the structure of basic sets of {Schur} rings over cyclic groups},
  year     = {1994},
  number   = {2},
  pages    = {655--678},
  volume   = {169},
  doi      = {10.1006/jabr.1994.1302},
  fjournal = {Journal of Algebra},
  language = {English},
  zbl      = {0810.20005},
  zbmath   = {682589},
}

@Article{MuzychukPonomarenko2009,
  author   = {Muzychuk, Mikhail E. and Ponomarenko, Ilia},
  journal  = {Eur. J. Comb.},
  title    = {Schur rings},
  year     = {2009},
  number   = {6},
  pages    = {1526--1539},
  volume   = {30},
  doi      = {10.1016/j.ejc.2008.11.006},
  fjournal = {European Journal of Combinatorics},
  language = {English},
  zbl      = {1195.20003},
  zbmath   = {5640334},
}

@Article{Schur1933,
  author   = {Schur, I.},
  journal  = {Sitzungsber. Preu{{\ss}}. Akad. Wiss., Phys.-Math. Kl.},
  title    = {Zur {Theorie} der einfach transitiven {Permutationsgruppen}},
  year     = {1933},
  pages    = {598--623},
  volume   = {1933},
  fjournal = {Sitzungsberichte der Preu{{\ss}}ischen Akademie der Wissenschaften, Physikalisch-Mathematische Klasse},
  language = {German},
  zbl      = {0007.14903},
  zbmath   = {3009780},
}

@Article{WeisfeilerLehman1968,
  author  = {Weisfeiler, Boris and Lehman, A. A.},
  journal = {Nauchno-tekhn Inform},
  title   = {A reduction of a graph to a canonical form and an algebra arising during this reduction},
  year    = {1968},
  number  = {9},
  pages   = {12--16},
  volume  = {2},
}

@Book{Wielandt1964,
  author       = {Wielandt, Helmut},
  publisher    = {Academic Press},
  title        = {Finite permutation groups},
  year         = {1964},
  address      = {New York and London},
  language     = {English},
  zbl          = {0138.02501},
  zbmath       = {3223737},
}

@Article{Wiertelak1978,
  author   = {Wiertelak, K.},
  journal  = {Acta Arith.},
  title    = {On the density of some sets of primes. {I}},
  year     = {1978},
  pages    = {183--196},
  volume   = {34},
  doi      = {10.4064/aa-34-3-183-196},
  fjournal = {Acta Arithmetica},
  language = {English},
  zbl      = {0373.10029},
  zbmath   = {3580638},
}

@Article{Wiertelak1978a,
  author   = {Wiertelak, K.},
  journal  = {Acta Arith.},
  title    = {On the density of some sets of primes. {II}},
  year     = {1978},
  pages    = {197--210},
  volume   = {34},
  doi      = {10.4064/aa-34-3-197-210},
  fjournal = {Acta Arithmetica},
  language = {English},
  zbl      = {0373.10030},
  zbmath   = {3580639},
}

@Book{Zieschang2005,
  author    = {Zieschang, Paul-Hermann},
  publisher = {Berlin: Springer},
  title     = {Theory of association schemes},
  year      = {2005},
  series    = {Springer Monogr. Math.},
  doi       = {10.1007/3-540-30593-9},
  fseries   = {Springer Monographs in Mathematics},
  language  = {English},
  zbl       = {1079.05099},
  zbmath    = {2206388},
}

@Book{ChenPonomarenko2019,
  author    = {Gang Chen and Ilia Ponomarenko},
  publisher = {Central China Normal University Press},
  title     = {Coherent Configurations},
  year      = {2019},
  address   = {Wuhan},
  note      = {Updated online version of the lecture notes},
  url       = {http://www.pdmi.ras.ru/~inp/ccNOTES.pdf},
}

@Article{Wiertelak2000,
  author   = {Wiertelak, K.},
  journal  = {Funct. Approximatio, Comment. Math.},
  title    = {On the density of some sets of primes {{\(p\)}}, for which {{\(n\mid \text{ord}_p a\)}}},
  year     = {2000},
  pages    = {237--241},
  volume   = {28},
  doi      = {10.7169/facm/1538186700},
  fjournal = {Functiones et Approximatio. Commentarii Mathematici},
  language = {English},
  zbl      = {1009.11056},
  zbmath   = {1599295},
}

\end{document}